\documentclass[reqno]{amsart}
\usepackage[left=1in, right=1in, top=1in]{geometry}

\usepackage{etoolbox}

\makeatletter
\let\old@tocline\@tocline
\let\section@tocline\@tocline
\newcommand{\section@dotsep}{4.5}
\newcommand{\subsection@dotsep}{4.5}
\patchcmd{\@tocline}
  {\hfil}
  {\nobreak
     \leaders\hbox{$\m@th
        \mkern \section@dotsep mu\hbox{.}\mkern \section@dotsep mu$}\hfill
     \nobreak}{}{}
\let\section@tocline\@tocline
\let\@tocline\old@tocline

\patchcmd{\@tocline}
  {\hfil}
  {\nobreak
     \leaders\hbox{$\m@th
        \mkern \subsection@dotsep mu\hbox{.}\mkern \subsection@dotsep mu$}\hfill
     \nobreak}{}{}
\let\subsection@tocline\@tocline
\let\@tocline\old@tocline

\let\old@l@section\l@section
\let\old@l@subsection\l@subsection

\def\@tocwriteb#1#2#3{%
  \begingroup
    \@xp\def\csname #2@tocline\endcsname##1##2##3##4##5##6{%
      \ifnum##1>\c@tocdepth
      \else \sbox\z@{##5\let\indentlabel\@tochangmeasure##6}\fi}%
    \csname l@#2\endcsname{#1{\csname#2name\endcsname}{\@secnumber}{}}%
  \endgroup
  \addcontentsline{toc}{#2}%
    {\protect#1{\csname#2name\endcsname}{\@secnumber}{#3}}}%

\newlength{\@tocsectionindent}
\newlength{\@tocsubsectionindent}
\newlength{\@tocsubsubsectionindent}
\newlength{\@tocsectionnumwidth}
\newlength{\@tocsubsectionnumwidth}
\newlength{\@tocsubsubsectionnumwidth}
\newcommand{\settocsectionnumwidth}[1]{\setlength{\@tocsectionnumwidth}{#1}}
\newcommand{\settocsubsectionnumwidth}[1]{\setlength{\@tocsubsectionnumwidth}{#1}}
\newcommand{\settocsubsubsectionnumwidth}[1]{\setlength{\@tocsubsubsectionnumwidth}{#1}}
\newcommand{\settocsectionindent}[1]{\setlength{\@tocsectionindent}{#1}}
\newcommand{\settocsubsectionindent}[1]{\setlength{\@tocsubsectionindent}{#1}}
\newcommand{\settocsubsubsectionindent}[1]{\setlength{\@tocsubsubsectionindent}{#1}}

\renewcommand{\l@section}{\section@tocline{1}{\@tocsectionvskip}{\@tocsectionindent}{}{\@tocsectionformat}}%
\renewcommand{\l@subsection}{\subsection@tocline{1}{\@tocsubsectionvskip}{\@tocsubsectionindent}{}{\@tocsubsectionformat}}%
\renewcommand{\l@subsubsection}{\subsubsection@tocline{1}{\@tocsubsubsectionvskip}{\@tocsubsubsectionindent}{}{\@tocsubsubsectionformat}}%
\newcommand{\@tocsectionformat}{}
\newcommand{\@tocsubsectionformat}{}
\newcommand{\@tocsubsubsectionformat}{}
\expandafter\def\csname toc@1format\endcsname{\@tocsectionformat}
\expandafter\def\csname toc@2format\endcsname{\@tocsubsectionformat}
\expandafter\def\csname toc@3format\endcsname{\@tocsubsubsectionformat}
\newcommand{\settocsectionformat}[1]{\renewcommand{\@tocsectionformat}{#1}}
\newcommand{\settocsubsectionformat}[1]{\renewcommand{\@tocsubsectionformat}{#1}}
\newcommand{\settocsubsubsectionformat}[1]{\renewcommand{\@tocsubsubsectionformat}{#1}}
\newlength{\@tocsectionvskip}
\newcommand{\settocsectionvskip}[1]{\setlength{\@tocsectionvskip}{#1}}
\newlength{\@tocsubsectionvskip}
\newcommand{\settocsubsectionvskip}[1]{\setlength{\@tocsubsectionvskip}{#1}}
\newlength{\@tocsubsubsectionvskip}
\newcommand{\settocsubsubsectionvskip}[1]{\setlength{\@tocsubsubsectionvskip}{#1}}

\patchcmd{\tocsection}{\indentlabel}{\makebox[\@tocsectionnumwidth][l]}{}{}
\patchcmd{\tocsubsection}{\indentlabel}{\makebox[\@tocsubsectionnumwidth][l]}{}{}
\patchcmd{\tocsubsubsection}{\indentlabel}{\makebox[\@tocsubsubsectionnumwidth][l]}{}{}

\newcommand{\@sectypepnumformat}{}
\renewcommand{\contentsline}[1]{%
  \expandafter\let\expandafter\@sectypepnumformat\csname @toc#1pnumformat\endcsname%
  \csname l@#1\endcsname}
\newcommand{\@tocsectionpnumformat}{}
\newcommand{\@tocsubsectionpnumformat}{}
\newcommand{\@tocsubsubsectionpnumformat}{}
\newcommand{\setsectionpnumformat}[1]{\renewcommand{\@tocsectionpnumformat}{#1}}
\newcommand{\setsubsectionpnumformat}[1]{\renewcommand{\@tocsubsectionpnumformat}{#1}}
\newcommand{\setsubsubsectionpnumformat}[1]{\renewcommand{\@tocsubsubsectionpnumformat}{#1}}
\renewcommand{\@tocpagenum}[1]{%
  \hfill {\mdseries\@sectypepnumformat #1}}

\let\oldappendix\appendix
\renewcommand{\appendix}{%
  \leavevmode\oldappendix%
  \addtocontents{toc}{%
    \protect\settowidth{\protect\@tocsectionnumwidth}{\protect\@tocsectionformat\sectionname\space}%
    \protect\addtolength{\protect\@tocsectionnumwidth}{2em}}%
}
\makeatother

\makeatletter
\settocsectionnumwidth{2em}
\settocsubsectionnumwidth{2.5em}
\settocsubsubsectionnumwidth{3em}
\settocsectionindent{1pc}%
\settocsubsectionindent{\dimexpr\@tocsectionindent+\@tocsectionnumwidth}%
\settocsubsubsectionindent{\dimexpr\@tocsubsectionindent+\@tocsubsectionnumwidth}%
\makeatother

\settocsectionvskip{0pt}
\settocsubsectionvskip{0pt}
\settocsubsubsectionvskip{0pt}

\usepackage{mathrsfs}
\usepackage{mathrsfs}
\usepackage{amsfonts}
\usepackage{comment}
\usepackage{cases}
\usepackage{latexsym}
\usepackage{amsmath}
\usepackage[all]{xy}
\usepackage{stmaryrd}
\usepackage{amsfonts}
\usepackage{amsmath,amssymb,amscd,bbm,amsthm,mathrsfs,dsfont}
\usepackage[notref, notcite]{showkeys}

\usepackage{tikz}

\usepackage{pgflibraryarrows}
\usepackage{pgflibrarysnakes}

\usepackage[numbers,sort&compress]{natbib}
\usepackage[pagebackref]{hyperref}
\usepackage{hypernat}

\usepackage{color, hyperref}
\definecolor{blue}{rgb}{0.9,0.0,0.9}
\hypersetup{colorlinks, breaklinks,
            linkcolor=red,urlcolor=blue,
            anchorcolor=blue,citecolor=blue}

\usepackage{fancyhdr}
\usepackage{amsxtra,ifthen}
\usepackage{verbatim}

\usepackage{fancyhdr}
\usepackage{amsxtra,ifthen}
\usepackage{verbatim}

\numberwithin{equation}{section}

\theoremstyle{plain}
\newtheorem{theorem}{Theorem}[section]
\newtheorem{lemma}[theorem]{Lemma}

\newtheorem{corollary}[theorem]{Corollary}

\theoremstyle{definition}

\allowdisplaybreaks[4]

\begin{document}

\title[Lie Biderivations on Triangular Algebras]
{Lie Biderivations on Triangular Algebras}

\author{Xinfeng Liang, Dandan Ren and Feng Wei}

\address{Liang: School of Mathematics and Statistics, AnHui university of science \& technology, 232001, Huainan, P.R. China}

\email{xfliang@aust.edu.cn}

\address{Ren: School of Mathematics and Statistics, AnHui university of science \& technology, 232001, Huainan, P.R. China}

\email{dandanren0225@163.com}

\address{Wei: School of Mathematics and Statistics, Beijing Institute of Technology, 100081, Beijing, P. R. China}

\email{daoshuo@hotmail.com\\
daoshuowei@gmail.com}

\begin{abstract}
Let $\mathcal{T}$ be a triangular algebra over a commutative ring $\mathcal{R}$ and
$\varphi: \mathcal{T} \times \mathcal{T}\longrightarrow \mathcal{T}$ be an arbitrary Lie biderivation of $\mathcal{T}$.
We will address the question of describing the form of $\varphi$ in the current work.
It is shown that under certain mild assumptions, $\varphi$
is the sum of an inner biderivation and an extremal biderivation and a some central bilinear mapping.
Our results is immediately applied to block upper triangular algebras and Hilbert space nest algebras .
\end{abstract}

\date{\today}



\thanks{The work of the first author is partially supported by
Key Projects of Natural Science Research in Anhui Province (Grant No. KJ2019A0107,KJ2018A0082);
the second author is partially supported by the National Natural Science Foundation
of China (11801008), Key Program of Scientific Research Fund
for Young Teachers of AUST (Grant No.QN2017209), 
talent Introduction Project of Anhui University of Science \& Technology(Grant No. 11690).}

\maketitle


\section{Introduction}\label{xxsec1}

Let $\mathcal{R}$ be a commutative ring with identity and $\mathcal{A}$ be an associative $\mathcal{R}$-algebra with center $\mathcal{Z(A)}$.
An $\mathcal{R}$-linear mapping $d: \mathcal{A}\longrightarrow \mathcal{A}$ is called a \textit{derivation}
if $d(xy)=d(x)y+xd(y)$ for all $x,y\in \mathcal{A}$,  and is called a \textit{Lie derivation} if
$$
d([x,y])=[d(x),y]+[x,d(x)]
$$
for all $x\in \mathcal{A}$. An $\mathcal{R}$-linear mapping $d: \mathcal{A}\longrightarrow \mathcal{A}$ of the form $a\mapsto am-ma $ for some $m\in \mathcal{A}$,
is said to be an \textit{inner derivation}.
A $\mathcal{R}$-bilinear mapping $\varphi: \mathcal{A}\times \mathcal{A}\rightarrow \mathcal{A}$
is a \textit{biderivation} if it is a derivation
with respect to both components, that is
$$
\varphi(xz, y)=\varphi(x,y)z+x\varphi(z,y)~\text{and}~\varphi(x,yz)=\varphi(x,y)z+y\varphi(x,z)
$$
for all $x,y\in \mathcal{A}$. If the algebra $\mathcal{A}$ is noncommutative,
then the mapping $\varphi(x, y)=\lambda[x, y]$ for all $x,y\in \mathcal{A}$ and
some $\lambda\in \mathcal{Z(A)}$ is called an \textit{inner biderivation}.
An $\mathcal{R}$-bilibear mapping $\varphi: \mathcal{A}\times \mathcal{A}\rightarrow \mathcal{A}$
is said to be an \textit{extremal biderivation}
if it is of the form $\varphi(x,y)=[x,[y, a]]$ for all $x, y\in \mathcal{A}$ and some $a\in \mathcal{A}, a\notin \mathcal{Z(A)}$.
An $\mathcal{R}$-bilinear mapping $\varphi: \mathcal{A}\times \mathcal{A}\rightarrow \mathcal{A}$
is a \textit{Lie biderivation} if it is a Lie derivation
with respect to both components, implying that
$$
\begin{aligned}
\varphi([x,z],y)&=[\varphi(x,y),z]+[x,\varphi(z,y)]~\text{and}~\\
\varphi(x,[y,z])&=[\varphi(x,y),z]+[y,\varphi(x,z)]
\end{aligned}
$$
for all $x,y\in \mathcal{A}$.

Suppose that
$A$ and $B$ are two unital algebras over $\mathcal{R}$ and $M$ is a
nonzero faithful bimodule as a left $A$-module and also right $B$-module. Then one can define
$$
\left[
\begin{array}
[c]{cc}%
A & M\\
0 & B\\
\end{array}
\right]=\left\{ \hspace{2pt} \left[
\begin{array}
[c]{cc}%
a & m\\
0 & b\\
\end{array}
\right] \hspace{2pt} \vline \hspace{2pt} a\in A, b\in B, m\in M
\hspace{2pt} \right\}
$$
to be an associative algebra under matrix-like addition and
matrix-like multiplication. An algebra $\mathcal{T}$ is called a
\textit{triangular algebra} if there exist algebras $A, B$ and
nonzero faithful $(A, B)$-bimodule $M$ such that $\mathcal{T}$ is
(algebraically) isomorphic to
$$
\left[
\begin{array}
[c]{cc}%
A & M\\
O & B\\
\end{array}
\right]
$$
under matrix-like addition and matrix-like multiplication. Usually,
we denote a triangular algebra by $\mathcal{T}=\left[\smallmatrix A & M\\
O & B \endsmallmatrix \right]$. This kind of algebra was first
introduced by Chase in \cite{Chase}. He applied triangular algebras
to show us the asymmetric behavior of semi-hereditary rings and
constructed an classical example of a left semi-hereditary ring
which is not right semi-hereditary. Harada referred to the
triangular algebras as generalized triangular matrix rings in the
literature \cite{Harada} which he used triangular algebras to study
the structure of hereditary semi-primary rings. The definition of
triangular algebra is somewhat formal and hence they are said to be
formal triangular matrix algebras in the situation of noncommutative
algebras \cite{HaghanyVara}.

The concept of biderivation originates from difference or functional equations
and their inequalities instead of associative algebras.
It was Maksa \cite{Maksa1980} who initially introduced the concept of biderivations.
Vukman \cite{Vukman1989, Vukman1990} investigated biderivations
on prime and semiprime rings. Bre\v sar et al. \cite{BresarMartindaleMiers1993} have shown
that each biderivation $\varphi$ on a noncommutative prime ring $\mathcal{R}$
is of the form $\varphi(x, y)=\lambda [x, y]$, for some element $\lambda$ in the extended centroid of $\mathcal{R}$.
It has turned out that this result can be applied to the problem of describing the form of commuting mappings.
We encourage the reader to refer to the survey paper \cite{Bresar2004} where applications of biderivations
to some other areas are provided.
The study of commuting mappings and biderivations has its deep roots in the structure theory of associative algebras,
where it has proved to be influential and far-reaching, see \cite{Bresar2004} and references therein.
 They have been becoming an active research topic
in the theory of additive mappings of associative algebras
since Bre\v sar's elegant work \cite{BresarMartindaleMiers1993, Bresar1995}.
On the other hand, an interest in studying these mappings on Lie algebras has been increasing more recently, see
\cite{BresarZhao2018, LiuGuoZhao2018, Tang2018}.

The objective of this paper is to investigate Lie biderivations on triangular algebras.
Many authors have made important and essantial contributions to the related topics,
see \cite{AbdiogluLee2017, Ahmed2016, Benkovic2009,
BenkovicEremita2004, Bresar1995, BresarZhao2018, CheraghpourGhosseiri2019, DuWang2013,
Eremita2017, Fosner2015, Ghosseiri2013, Ghosseiri2017, LiangRenWei2019, LiuGuoZhao2018,
MGRSO2017, Tang2018, Wang2016, ZhangFengLiWu2006, ZhaoWangYao2009}. Cheung in \cite{Cheung2000} initiated the study of linear
mappings of abstract triangular algebras and obtained a number of elegant results. He gave detailed descriptions concerning
automorphisms, derivations, commuting mappings and Lie derivations of triangular algebras in \cite{Cheung2000, Cheung2001}.
Benkovi\v c \cite{Benkovic2005} considered Jordan derivations of triangular matrices over a commutative ring with
identity and proved that any Jordan derivation from the algebra of all upper triangular matrices into its arbitrary bimodule is the sum
of a derivation and an antiderivation.  Zhang and Yu \cite{ZhangYu2006} observed that each Jordan derivation on a
2-torsion free triangular algebra is a derivation. Generalized Biderivations on nest algebras were also studied by Zhang et al. \cite{ZhangFengLiWu2006}.
They provided a sufficient and necessary condition which enable each generalized biderivation
on a complex separable Hilbert space nest algebra to be inner. Zhao et al. \cite{ZhaoWangYao2009} investigated
biderivations on upper triangular matrix algebras over a commutative ring $\mathcal{R}$. They proved that
each biderivation on the algebra $\mathcal{T}_n(\mathcal{R})$ of all upper triangular $n\times n$ matrices over
$\mathcal{R}$ is the sum of an inner biderivation and an extremal biderivation.

Benkovi\v c  \cite{Benkovic2009} paid special attention to biderivations on a certain class of triangular algebras. He obtained
that a bilinear biderivation $\varphi$ of a triangular algebra $\mathcal{T}$ satisfying certain conditions (see the conditions (1)--(4) in Theorem \ref{xxsec3.2}) is of the form
$\varphi(x, y)=\lambda[x, y]+[x, [y, r]]$ for some element $\lambda\in \mathcal{Z(T)}$ and some element $r\in \mathcal{T}$.
On the other hand, Ghosseiri \cite{Ghosseiri2013} considered biderivations of an arbitrary triangular ring $\mathcal{T}$
(not assuming $M$ is a faithful $(A, B)$-bimodule). He proved that each biderivation $\varphi: \mathcal{T}\times \mathcal{T}\longrightarrow \mathcal{T}$
can be decomposed into $\varphi= \tau+\psi+\delta$, where $\tau$ is a biderivation satisfying certain conditions,
$\psi$ is an extremal biderivation, $\delta$ is a special kind of biderivation.
Basing on previous Benkovi\v c's and Ghosseiri's works, our\cite{LiangRenWei2019} use the same as condition (see the conditions (1)--(4) in Theorem \ref{xxsec3.4}) to obtain the form of Jordan biderivation of a triangular algebras $\mathcal{T}$;  Eremita \cite{Eremita2017} used the notion of
the maximal left ring of quotients to describe the form of biderivations of a triangular ring. His distinguished approach
permit him to achieve double purpose:  generalizing Benkovi\v c's result on biderivations \cite[Theorem 4.11]{Benkovic2009}  and
refining Ghosseiri's result on biderivations \cite[Theorem 2.4]{Ghosseiri2013}.
More recently, Ghosseiri and his collaborators \cite{CheraghpourGhosseiri2019, Ghosseiri2017} further
characterized the structure of biderivations and superderivations on trivial extensions,
which are natural generalizations of the corresponding results on triangular algebras.
Wang et al\cite{WAngYuChen2011} investigated biderivations on Parabolic Subalgebras of Simple Lie Algebras $\mathfrak{g}$ of rank l over an algebraically closed field of characteristic zero. Let $\mathfrak{p}$ an arbitrary parabolic subalgebra of $\mathfrak{g}$, they proved that a bilinear map
$\varphi:\mathfrak{p}\times \mathfrak{p}\rightarrow \mathfrak{p}$ is a biderivation if and only if it is a sum of an inner and an extremal biderivation. Wang and Yu\cite{WangYu2013} observed that each biderivation of
Schr\"{0}dinger-Virasoro Lie algebra $\Omega$ over the complex field $\mathfrak{C}$ is inner.
As an application of biderivations, they show that every linear commuting map $\varphi$ on $\Omega$ has the form $\varphi(x)=\lambda x+\mathfrak{f}(x)M_0$, where $\lambda\in \mathfrak{C}$, $M_0$ is a basis of the one-dimensional center of $\Omega$, and $\mathfrak{f}$ is a linear function from $\Omega$ to $\mathfrak{C}$.
Cheng et al. \cite{ChengWangSunZhang2017} prove that each skew-symmetric biderivation of the Lie algebra $\mathfrak{gca}$ over the complex field $\mathfrak{C}$ is inner. As an application of biderivations, we will show that every linear commuting map   $\phi$ on the Lie algebra  $\mathfrak{gca}$  has the form $\phi(x)=\lambda x$ , where $\lambda\in \mathfrak{C}$.
Liu et al.\cite{LiuGuoZhao2018} determine the biderivations of the block Lie algebras $\mathcal{B}(q)$ for all $q\in \mathfrak{C}$. More precisely, they prove that the space of biderivations of $\mathcal{B}(q)$ is spanned by inner biderivations and one outer biderivation. Applying this result,  they also research all commuting maps on $\mathcal{B}(q)$.

This paper is devoted to the treatment of Lie biderivations of
triangular algebras, and its framework is as follows. After Introduction, the second
section states some fundamental facts about triangular algebras and demonstrates three classical examples.
The kernel question of the current work is to describe the decomposition forms of Lie biderivations on triangular algebras, which
takes places in the mainbody--Section 3.

\section{Preliminaries}\label{xxsec2}

\vspace{2mm}

Let $\mathcal{R}$ be a commutative ring with identity. Let $A$ and $B$
be unital algebras over $\mathcal{R}$. Recall that an $(A, B)$-bimodule $M$ is
\textit{faithful} if for any $a\in
A$ and $b\in B$, $aM=0$ (resp. $Mb=0$) implies that $a=0$ (resp. $b=0$).

Let $A, B$ be unital associative algebras over $\mathcal{R}$ and $M$
be a unital $(A,B)$-bimodule, which is faithful as a left $A$-module
and also as a right $B$-module. We denote the {\em triangular algebra}
consisting of $A, B$ and $M$ by
$$
\mathcal{T}=\left[
\begin{array}
[c]{cc}%
A & M\\
O & B\\
\end{array}
\right] .
$$
Then $\mathcal{T}$ is an associative and noncommutative
$\mathcal{R}$-algebra. The center $\mathcal{Z(T)}$ of $\mathcal{T}$ is (see \cite[Proposition 3]{Cheung2003})
$$
\mathcal{Z(T)}=\left\{ \left[
\begin{array}
[c]{cc}%
a & 0\\
0 & b
\end{array}
\right] \vline \hspace{3pt} am=mb,\ \forall\ m\in M \right\}.
$$
Let us define two natural $\mathcal{R}$-linear projections
$\pi_A:\mathcal{T}\rightarrow A$ and $\pi_B:\mathcal{T}\rightarrow
B$ by
$$
\pi_A: \left[
\begin{array}
[c]{cc}%
a & m\\
0 & b\\
\end{array}
\right] \longmapsto a \quad \text{and} \quad \pi_B: \left[
\begin{array}
[c]{cc}%
a & m\\
0 & b\\
\end{array}
\right] \longmapsto b.
$$
It is easy to see that $\pi_A \left(\mathcal{Z(T)}\right)$ is a
subalgebra of ${\mathcal Z}(A)$ and that $\pi_B(\mathcal{Z(T)})$ is
a subalgebra of ${\mathcal Z}(B)$. Furthermore, there exists a
unique algebraic isomorphism $\tau\colon
\pi_A(\mathcal{Z(T)})\longrightarrow \pi_B(\mathcal{Z(T)})$ such
that $am=m\tau(a)$ for all $a\in \pi_A(\mathcal{Z(T)})$ and for all
$m\in M$.

Let $1$ (resp. $1^\prime$) be the identity of the algebra $A$ (resp.
$B$), and let $I$ be the identity of the triangular algebra
$\mathcal{T}$. We will use the following notations:
$$
e=\left[
\begin{array}
[c]{cc}%
1 & 0\\
0 & 0\\
\end{array}
\right], \hspace{8pt} f=I-e=\left[
\begin{array}
[c]{cc}%
0 & 0\\
0 & 1^\prime\\
\end{array}
\right]
$$
and
$$
\mathcal{T}_{11}=e{\mathcal T}e, \hspace{6pt}
\mathcal{T}_{12}=e{\mathcal T}f, \hspace{6pt}
\mathcal{T}_{22}=f{\mathcal T}f.
$$
Thus the triangular algebra $\mathcal{T}$ can be written as
$$
\mathcal{T}=e{\mathcal T}e+e{\mathcal T}f+f{\mathcal T}f
=\mathcal{T}_{11}+\mathcal{T}_{12}+\mathcal{T}_{22}.
$$
Here, $\mathcal{T}_{11}$ and $\mathcal{T}_{22}$ are subalgebras of
$\mathcal{T}$ which are isomorphic to $A$ and $B$, respectively.
$\mathcal{T}_{12}$ is a $(\mathcal{T}_{11},
\mathcal{T}_{22})$-bimodule which is isomorphic to the $(A,
B)$-bimodule $M$. It should be remarked that $\pi_A(\mathcal{Z(T)})$
and $\pi_B(\mathcal{Z(T)})$ are isomorphic to $e\mathcal{Z(T)}e$ and
$f\mathcal{Z(T)}f$, respectively. Then there is an algebra
isomorphism $\tau\colon e\mathcal{Z(T)}e\longrightarrow
f\mathcal{Z(T)}f$ such that $am=m\tau(a)$ for all $m\in
e\mathcal{T}f$.

\vspace{3mm}

Let us see several classical triangular algebras which will be frequently invoked
in the sequel discussion.

\subsection{Upper triangular matrix algebras}
\label{xxsec2.1}

Let $\mathcal{R}$ be a commutative ring with identity. We denote the
set of all $p\times q$ matrices over $\mathcal{R}$ by $M_{p\times
q}(\mathcal{R})$ and denote the set of all $n\times n$ upper
triangular matrices over $\mathcal{R}$ by $T_n(\mathcal{R})$. For
$n\geq 2$ and each $1\leq k \leq n-1$, the \textit{upper triangular
matrix algebra} $T_n(\mathcal{R})$ can be written as
$$
T_n(\mathcal{R})=\left[
\begin{array}
[c]{cc}%
T_k(\mathcal{R}) & M_{k\times (n-k)}(\mathcal{R})\\
O & T_{n-k}(\mathcal{R})
\end{array}
\right] .
$$

\subsection{Block upper triangular matrix algebras}
\label{xxsec2.2}

Let $\mathcal{R}$ be a commutative ring with identity. For each
positive integer $n$ and each positive integer $m$ with $m\leq n$,
we denote by $\bar{d}=(d_1, \cdots, d_i, \cdots, d_m)\in
\mathbb{N}^m$ an ordered $m$-vector of positive integers such that
$n=d_1+\cdots +d_i+\cdots+d_m$. The \textit{block upper triangular
matrix algebra} $B^{\bar{d}}_n(\mathcal{R})$ is a subalgebra of
$M_n(\mathcal{R})$ of the form
$$
B^{\bar{d}}_n(\mathcal{R})=\left[
\begin{array}
[c]{ccccc}%
M_{d_1}(\mathcal{R})  & \cdots & M_{d_1\times d_i}(\mathcal{R}) &
\cdots & M_{d_1\times
d_m}(\mathcal{R})\\
& \ddots & \vdots &  & \vdots  \\
 & & M_{d_i}(\mathcal{R}) & \cdots & M_{d_i\times d_m}(\mathcal{R}) \\
 & O  &  & \ddots & \vdots \\
 &  &  & & M_{d_m}(\mathcal{R})  \\
\end{array}
\right] =
$$
$$
\left[\smallmatrix \boxed{\smallmatrix
r_{1,1} & \cdots & r_{1,d_1}\\
\vdots & \ddots & \vdots\\
r_{d_1,1}& \cdots & r_{d_1,d_1}\endsmallmatrix} & \cdots &
\boxed{\smallmatrix
r_{1, x+1} & \cdots & r_{1, x+d_i}\\
\vdots & \ddots & \vdots\\
r_{d_1,x+1}& \cdots & r_{d_1, x+d_i}
\endsmallmatrix} & \cdots &
\boxed{\smallmatrix
r_{1, y+1} & \cdots & r_{1, y+d_m}\\
\vdots & \ddots & \vdots\\
r_{d_1, y+1}& \cdots & r_{d_1, y+d_m}
\endsmallmatrix}\\
 & \ddots & \vdots \hspace{48pt}\vdots & & \vdots \hspace{48pt}\vdots \\
 &  &
 \boxed{\smallmatrix
r_{x+1,x+1} & \cdots & r_{x+1, x+d_i}\\
\vdots & \ddots & \vdots\\
r_{x+d_i, x+1}& \cdots & r_{x+d_i,x+d_i}
\endsmallmatrix} & \cdots &
\boxed{\smallmatrix
r_{x+1,y+1} & \cdots & r_{x+1,y+d_m}\\
\vdots & \ddots & \vdots\\
r_{x+d_i,y+1}& \cdots & r_{x+d_i, y+d_m}
\endsmallmatrix} \\
 &  &  & \ddots & \vdots \hspace{48pt}\vdots\\
 & O &  & &
 \boxed{\smallmatrix
r_{y+1, y+1} & \cdots & r_{y+1,y+d_m}\\
\vdots & \ddots & \vdots\\
r_{y+d_m,y+1}& \cdots & r_{y+d_m,y+d_m}
\endsmallmatrix}
\endsmallmatrix
\right] .
$$
Note that the full matrix algebra $M_n(\mathcal{R})$ of all $n\times
n$ matrices over $\mathcal{R}$ and the upper triangular matrix
algebra $T_n(\mathcal{R})$ of all $n\times n$ upper triangular
matrices over $\mathcal{R}$ are two special cases of block upper
triangular matrix algebras. If $n\geq 2$ and
$B^{\bar{d}}_n(\mathcal{R})\neq M_n(\mathcal{R})$, then
$B^{\bar{d}}_n(\mathcal{R})$ is a triangular algebra and can be
represented as
$$
B^{\bar{d}}_n(\mathcal{R})=\left[
\begin{array}
[c]{cc}%
B^{\bar{d}_1}_j(\mathcal{R}) & M_{j\times (n-j)}(\mathcal{R})\\
O_{(n-j)\times j} & B^{\bar{d}_2}_{n-j}(\mathcal{R})\\
\end{array}
\right],
$$
where $1\leq j < m$ and $\bar{d}_1\in \mathbb{N}^j, \bar{d}_2\in
\mathbb{N}^{m-j}$.

\subsection{Nest algebras}
\label{xxsec2.3}

Let $\mathbf{H}$ be a complex Hilbert space and
$\mathcal{B}(\mathbf{H})$ be the algebra of all bounded linear
operators on $\mathbf{H}$. Let $I$ be a index set. A \textit{nest}
is a set $\mathcal{N}$ of closed subspaces of $\mathbf{H}$
satisfying the following conditions:
\begin{enumerate}
\item[(1)] $0, \mathbf{H}\in \mathcal{N}$;
\item[(2)] If $N_1, N_2\in \mathcal{N}$, then
either $N_1\subseteq N_2$ or $N_2\subseteq N_1$;
\item[(3)] If $\{N_i\}_{i\in I}\subseteq \mathcal{N}$,
then $\bigcap_{i\in I}N_i\in \mathcal{N}$;
\item[(4)] If $\{N_i\}_{i\in I}\subseteq \mathcal{N}$, then the
norm closure of the linear span of $\bigcup_{i\in I} N_i$ also lies
in $\mathcal{N}$.
\end{enumerate}
If $\mathcal{N}=\{0, \mathbf{H}\}$, then $\mathcal{N}$ is called a
trivial nest, otherwise it is called a non-trivial nest.

The \textit{nest algebra} associated with $\mathcal{N}$ is the set
$$
\mathcal{T}(\mathcal{N})=\{\hspace{3pt} T\in
\mathcal{B}(\mathbf{H})\hspace{3pt}| \hspace{3pt} T(N)\subseteq N
\hspace{3pt} {\rm for} \hspace{3pt} {\rm all} \hspace{3pt} N\in
\mathcal{N}\} .
$$
A nontrivial nest algebra is a triangular algebra. Indeed, if $N\in
\mathcal{N}\backslash \{0, {\mathbf H}\}$ and $E$ is the orthogonal
projection onto $N$, then $\mathcal{N}_1=E(\mathcal{N})$ and
$\mathcal{N}_2=(1-E)(\mathcal{N})$ are nests of $N$ and $N^{\perp}$,
respectively. Moreover,
$\mathcal{T}(\mathcal{N}_1)=E\mathcal{T}(\mathcal{N})E,
\mathcal{T}(\mathcal{N}_2)=(1-E)\mathcal{T}(\mathcal{N})(1-E)$ are
nest algebras and
$$
\mathcal{T}(\mathcal{N})=\left[
\begin{array}
[c]{cc}%
\mathcal{T}(\mathcal{N}_1) & E\mathcal{T}(\mathcal{N})(1-E)\\
O & \mathcal{T}(\mathcal{N}_2)\\
\end{array}
\right].
$$
Note that any finite dimensional nest algebra is isomorphic to a
complex block upper triangular matrix algebra. We refer the reader
to \cite{Davidson1988} for the theory of nest algebras.

\subsection{Matrix Incidence algebras}
\label{xxsec2.4}

Let $\mathbb{K}$ be a field and $A$ be a unital algebra over
$\mathbb{K}$. Let $X$ be a partially ordered set with
the partial order $\leq$. We define the \textit{incidence algebra}
of $X$ over $A$ as
$$
I(X, A)=\{f: X\times X \longrightarrow A \hspace{2pt} | \hspace{2pt}
f(x,y)=0 \hspace{4pt} {\rm if} \hspace{4pt} x\nleq y\}
$$
with algebraic operation given by
\begin{align*}
(f+g)(x, y) & =f(x, y)+g(x, y),\\
(f * g)(x, y) & =\sum_{x \leq z\leq y}f(x, z)g(z, y), \\
(k \cdot f)(x, y)& =k \cdot f(x, y)
\end{align*}
for all $f, g\in I(X, A), k\in \mathbb{K}$ and $x, y, z\in X$.
Obviously, $f$ is a $A$-valued function on $\{(x,y)\in X\times X
\vert x\leq y\}$. The product $*$ is usually called
\textit{convolution} in function theory. In particular, if $X$ be finite partially ordered
set with $n$ elements, then $I(X, A)$ is isomorphic to a
subalgebra of the algebra $M_n(A)$ of square matrices over
$\mathbb{K}$ with elements $[a_{ij}]_{n\times n}\in M_n(A)$ satisfying
$a_{ij}=0$ if $i\nleq j$, for some partial order $\leq$ defined in
the partial order set (poset) $\{1, \cdots, n\}$ \cite[Proposition
1.2.4]{SpiegelDonnell1997}. More precisely, $I(X, A)$ is isomorphic to an
upper triangular matrix algebra with entries $A$ or 0.
We will call such incidence algebras \textit{matrix incidence algebras}.
In fact, any incidence algebra arising from a finite partially ordered set is isomorphic to
some matrix incidence algebra $I(X, A)$, where $\leq$ is consistent with the natural order.
Nevertheless, we can not say that each matrix incidence algebra is
a triangular algebra in which $M$ is a faithful $(A, B)$-bimodule in usual.
Not all incidence algebras meet this condition.
If $X$ is a finite partial ordered set which is connected, then each matrix incidence algebra
$I(X, A)$ can be considered as a triangular algebra.

To illustrate this conclusion, let us see an intuitional
example. Let $X=\{1, 2, 3, 4, 5, 6, 7\}$ be a partially ordered set
and its relations generated by
$$
\{1\leq 3, 2\leq 3, 3\leq 4, 4\leq 5, 5\leq 6, 5\leq 7\}.
$$
We represent this partially ordered set $X$ by the following diagram
$$
X=\left[
\begin{array}{c}
\xymatrix{ 1  \ar[dr] & & & & 6\\
& 3 \ar[r] & 4 \ar[r] & 5 \ar[ur] \ar[dr] & \\
2  \ar[ur] &  &  &  & 7 }
\end{array}
\right].
$$
Then we have
$$
I(X, A)\cong \left[
\begin{array}{ccccccc}
A & 0 & A & A & A & A & A\\
0 & A & A & A & A & A & A\\
0 & 0 & A & A & A & A & A\\
0 & 0 & 0 & A & A & A & A\\
0 & 0 & 0 & 0 & A & A & A\\
0 & 0 & 0 & 0 & 0 & A & 0\\
0 & 0 & 0 & 0 & 0 & 0 & A
\end{array}
\right] .
$$
The incidence algebra of a partially ordered set (poset) $X$ is the
algebra of functions from the segments of $X$ into an
$\mathbb{K}$-algebra $A$, which extends the various convolutions in
algebras of arithmetic functions. Incidence algebras, in fact, were
first considered by Ward \cite{Ward1937} as generalized algebras of
arithmetic functions. Rota and Stanley developed incidence algebras
as the fundamental structures of enumerative combinatorial theory
and allied areas of arithmetic function theory. The theory of
M\"{o}bius functions, including the classical M\"{o}bius function of
number theory and the combinatorial inclusion-exclusion formula, is
established in the context of incidence algebras. We refer to the
reader \cite{Stanley1997} for all these. On the other hand, the
algebraic properties of incidence algebras are quite striking as
well, including the fact that the lattice of ideals (in the
finite-dimensional case) is distributive, and that the partial order
can be recovered from the algebra. The latter has led to a complete
description of the automorphisms and derivations of the algebra
\cite{Stanley1970}.

In the theory of operator algebras, incidence algebras are refereed
to as ``bigraph algebras" or ``finite dimensional CSL algebras". For
a finite dimensional Hilbert space $\mathbf{H}$ and the algebra
$\mathcal{B}(\mathbf{H})$ of all bounded linear operators on
$\mathbf{H}$. A digraph algebra is a subalgebra $A$ of
$\mathcal{B}(\mathbf{H})$ which contains a maximal abelian
self-adjoint subalgebra $\mathcal{D}$ of $\mathcal{B}(\mathbf{H})$.
Since $\mathcal{D}$ is maximal abelian, the invariant projections
for $A$, ${\rm Lat}A$, are elements of $\mathcal{D}$ and so are
mutually commuting. Thus $A$ is a CSL-algebra (The abbreviation CSL
denotes `commutative subspace lattice'). Obviously, $A$ is finite
dimensional; on the other hand, every finite dimensional CSL-algebra
acts on a finite dimensional Hilbert space and contains a mass. The
term digraph algebra refers to the fact that associated with $A$
there is a directed graph on the set of vertices $\{1, 2, \cdots,
n\}$. This graph contains all the self loops. Then A contains the
matrix unit $e_{ij}$ if and only if there is a (directed) edge from
$j$ to $i$ in the digraph.

Let $I(X, A)$ be an incidence algebra of $X$ over $A$. The identity
element $\epsilon$ of $I(X, A)$ is given by $\epsilon(x,
y)=\delta_{xy}$ for all $x\leq y$, where $\delta_{xy}\in \{0, 1\}$
is the Kronecker sign. For each pair $x, y\in X$ with $x\leq y$ we
define $\epsilon_{xy}(u, v)=\delta_{xu}\delta_{yv}$ for all $u\leq
v$. Then $\epsilon_{xy}*\epsilon_{zu}=\delta_{yz}\epsilon_{xu}$ and
$\epsilon_{xy}a=a\epsilon_{xy}$ for all $a\in A$. Let $1\leq n \leq
\infty$, let $X=\{1, 2, \cdots, n\}$ if $n<\infty$, or $X=\{1, 2,
\cdots, \}$ if $n=\infty$, and endow $X$ with the usual linear
ordering. Then $I(X, A)$ can be identified with the upper triangular
matrix algebra $T_n(A)$ by identifying $\epsilon_{xy}$ with the
matrix $[\delta_{xi}\delta_{yj}]_{i, j=1}^n$. Note that the case
$T_{\infty}(A)$ is of infinite matrices. As another extreme case,
let $X=\{1, 2, \cdots, n\}$ with $1\leq n < \infty$. If $X$ has the
pre-order $\leq^\prime$, where $i\leq^\prime j$ for each pair $(i,
j)\in X\times X$, then $I(X, A)\cong M_n(A)$, the full matrix
algebras of $n\times n$ matrices over $A$. We now deduce some
results for the upper triangular matrix algebras $T_n(A)$ with
$1\leq n \leq \infty$ and for the full matrix algebras $M_n(A)$ with
$1\leq n < \infty$ from the results for incidence algebras.

\section{Lie Biderivations}
\label{xxsec3}

This part is the main part of our work, this part is mainly research Lie biderivation
of triangular algebras. In order to better explain our work, we prove the following important formulas.

\begin{lemma}\label{xxsec3.1}
Let $\mathcal{A}$ be a associative algebra over a commutative ring $\mathcal{R}$
 and $\phi: \mathcal{A}\times \mathcal{A}\rightarrow \mathcal{A}$ be a Lie biderivation on $\mathcal{A}$,
then $\phi$ has the fowwing properties:
$$
[\phi(x,a),[b,y]]+[\phi(x,b),[y,a]]=[\phi(y,a),[x,b]]+[\phi(y,b),[x,a]]
$$
for all $a,b,x,y\in \mathcal{A}$.
\end{lemma}

\begin{proof}
Let $\phi: \mathcal{A}\times \mathcal{A}\rightarrow \mathcal{A}$ be a Lie biderivation
of $\mathcal{A}$.
For arbitrary $x,y,a,b\in \mathcal{A}$, let us compute $\phi([x,y],[a,b])$.
Since $\phi$ is a Lie derivation with respective to the first component, we have
$$
\begin{aligned}
\phi([x,y],[a,b])&=[\phi(x,[a,b]),y]+[x,\phi(y,[a,b])]\\
&=[[a,\phi(x,b)]+[\phi(x,a),b],y]+[x,[\phi(y,a),b]+[a,\phi(y,b)]]\\
&=[a\phi(x,b)-\phi(x,b)a+\phi(x,a)b-b\phi(x,a),y]\\
&+[x,\phi(y,a)b-b\phi(y,a)+a\phi(y,b)-\phi(y,b)a]\\
&=(a\phi(x,b)-\phi(x,b)a+\phi(x,a)b-b\phi(x,a))y\\
&-y(a\phi(x,b)-\phi(x,b)a+\phi(x,a)b-b\phi(x,a))\\
&+x(\phi(y,a)b-b\phi(y,a)+a\phi(y,b)-\phi(y,b)a)\\
&-(\phi(y,a)b-b\phi(y,a)+a\phi(y,b)-\phi(y,b)a)x;
\end{aligned}    \eqno{(3.1)}
$$
Likewise, the mapping $\phi$ is a Lie biderivation with respect to
the second component as well, we have
$$
\begin{aligned}
\phi([x,y],[a,b])&=[\phi([x,y],a),b]+[a,\phi([x,y],b)]\\
&=[[\phi(x,a),y]+[x,\phi(y,a)],b]+[a,[\phi(x,b),y]+[x,\phi(y,b)]]\\
&=[\phi(x,a)y-y\phi(x,a)+x\phi(y,a)-\phi(y,a)x,b]\\
&+[a,\phi(x,b)y-y\phi(x,b)+x\phi(y,b)-\phi(y,b)x]]\\
&=(\phi(x,a)y-y\phi(x,a)+x\phi(y,a)-\phi(y,a)x)b\\
&-b(\phi(x,a)y-y\phi(x,a)+x\phi(y,a)-\phi(y,a)x)\\
&+a(\phi(x,b)y-y\phi(x,b)+x\phi(y,b)-\phi(y,b)x)\\
&-(\phi(x,b)y-y\phi(x,b)+x\phi(y,b)-\phi(y,b)x)a.
\end{aligned}   \eqno{(3.2)}
$$

Comparing $(3.1)$ and $(3.2)$, one can obtain
$$
[\phi(x,a),[b,y]]+[\phi(y,b),[a,x]]=[\phi(x,b),[a,y]]+[\phi(y,a),[x,b]]
$$
For arbitrary $x,y,a,b\in \mathcal{A}$.
\end{proof}

Let $M$ be a unital $(A,B)$-bimodule, the mapping $f:M\rightarrow M$
satisfying $f(am)=af(m)$ and $f(mb)=f(m)b$
for all $a\in A, m\in M, b\in B$ is called \text{bimodule homomorphism}.
A bimodule homomorphism $f:M\rightarrow M$
is of the \text{standard form} if there exist $a_0\in\mathcal{ Z}(A),b_0\in \mathcal{Z}(B)$ such that
$$
f(m)=a_0m+mb_0                        \eqno{(3.3)}
$$
for all $m\in M$.

Below we give the main theorem of this paper.

\begin{theorem}\label{xxsec3.2}
Let $\mathcal{T}=
\left[\smallmatrix A & M\\
O & B \endsmallmatrix \right]$ be a triangular algebra over a commutative ring $\mathcal{R}$
and let $\phi: \mathcal{T}\times \mathcal{T}\longrightarrow \mathcal{T}$
be a Lie biderivation. If the following conditions holds:
\begin{itemize}
  \item [(i)] $\pi_{A}(\mathcal{Z}(\mathcal{T}))=\mathcal{Z}(A)$ and  $\pi_{B}(\mathcal{Z}(\mathcal{T}))=\mathcal{Z}(B)$;
  \item [(ii)]at least one of the algebras $A$ and $B$ is noncommutative;
  \item [(iii)]each bimodule homomorphism $\mathfrak{f}:M\rightarrow M$ is of the standard form;
  \item [(iv)]if $\alpha a=0$, $\alpha\in \mathcal{Z}(A), 0\neq a\in A$, then $\alpha=0$.
\end{itemize}
Then every Lie biderivation $\phi: \mathcal{T}\times \mathcal{T}\rightarrow \mathcal{T}$ is of the form
$$
\phi(x,y)=\lambda_0[x,y]+[x,[y,\phi(e,e)]]+\mu(x,y)
$$
where for some $\lambda_0\in \mathcal{Z}(\mathcal{T})$ and
$\mu:\mathcal{T}\times \mathcal{T}\rightarrow \mathcal{Z}(\mathcal{T})$ is a central mapping
 for arbitrary $x,y\in \mathcal{T}$.
\end{theorem}

In order to better prove the main theorem,
we need to obtain the following series of lemmas.

\begin{lemma}\label{xxsec3.3}
Let $\mathcal{T}=
\left[\smallmatrix A & M\\
O & B \endsmallmatrix \right]$ be a triangular algebra over a commutative ring $\mathcal{R}$
and let $\phi: \mathcal{T}\times \mathcal{T}\longrightarrow \mathcal{T}$
be a Lie biderivation.
\begin{enumerate}
  \item[(1)] $\phi(0,x)=\phi(x,0)=0$;
  \item[(2)] $\phi(1,x)=e\phi(1,x)e \oplus f\phi(1,x)f\in \mathcal{Z}(\mathcal{T})$ ~\text{and}~
             $\phi(x,1)=e\phi(x,1)e \oplus f\phi(x,1)f\in \mathcal{Z}(\mathcal{T})$ ;
  \item[(3)] $e\phi(e,e)f=-e\phi(f,e)f=-e\phi(e,f)f=e\phi(f,f)f$.
\end{enumerate}
for all $x\in \mathcal{T}$
\end{lemma}

\begin{proof}
  \item[(1)]
Since $\phi$ is a Lie derivation with respect to the first component, we have
$$
\begin{aligned}
\phi(0,x)&=\phi([0,0],x)\\
&=[0,\phi(0,x)]+[\phi(0,x),0]=0
\end{aligned} \eqno{(3.4)}
$$
for all $x\in \mathcal{T}$.

  \item[(2)]
Since $\phi$ is a Lie derivation with respect to the first component and also
use the relation $(3.4)$, we have
$$
\begin{aligned}
0=\phi(0,x)&=\phi([1,y],x)\\
&=[\phi(1,x),y]+[1,\phi(y,x)]=[\phi(1,x),y],
\end{aligned}
$$
for arbitrary $x, y\in \mathcal{T}$. Because of the arbitrariness of element $y\in \mathcal{T}$,
one can obtain $\phi(1,x)\in \mathcal{Z}(\mathcal{T})$.
Furthermore, we have
By an analogous manner of $\phi(1,x)$, we have $\phi(x,1)\in \mathcal{Z}(\mathcal{T})$ for all $x\in \mathcal{T}$.

  \item[(3)]
In view of $(2)$, one can obtain $e\phi(1,x)f=0$ for all $x\in \mathcal{T}$.
Further using the relation $e+f=1$ and taking $x=e$ and $x=f$ respectively, we get
$$
e\phi(e,e)f=-e\phi(f,e)f \ \ \text{and} \ \ e\phi(e,f)f=-e\phi(f,f)f . \eqno{(3.5)}
$$
By an analogous manner, according to the relation $(2)$, we obtain
$e\phi(x,1)f=0$ for all $x\in \mathcal{T}$. Furthermore, we have 
$$
e\phi(e,e)f=-e\phi(e,f)f \ \ \text{and} \ \ e\phi(f,e)f=-e\phi(f,f)f. \eqno{(3.6)}
$$
Comparing $(3.5)$ with $(3.6)$, we have
$$
e\phi(e,e)f=-e\phi(f,e)f=-e\phi(e,f)f=e\phi(f,f)f.
$$
\end{proof}

\begin{lemma}\label{xxsec3.4}
With notations as above, we have
\begin{enumerate}
  \item [(1)]$\phi(a,m)=\alpha_0am=-\phi(m,a)$;
  \item [(2)]$\phi(b,m)=\alpha_0mb=-\phi(m,b)$
\end{enumerate}
for all $a\in A, b\in B, m\in M$.
\end{lemma}

\begin{proof}
Since $\phi$ is a Lie derivation with respect to the second component, we have
$$
\begin{aligned}
\phi(a,m)=\phi(a,[e,m])&=[\phi(a,e),m]+[e,\phi(a,m)]\\
&=\phi(a,e)m-m\phi(a,e)+e\phi(a,m)-\phi(a,m)e
\end{aligned}
$$
for all $a\in A, m\in M$. Multiplying the above equation by $e$ on the left side
and by $e$ on the right side, we have $e\phi(a,m)e=0$. Similarly, one can obtain $f\phi(a,m)f=0$ and
$$
e\phi(a,e)m=m\phi(a,e)f,      \eqno(3.7)
$$
for all $a\in A, m\in M$. Based on the above three formulas, we can get
$$
\phi(a,m)=e\phi(a,m)f      \eqno(3.8)
$$
for all $a\in A, \in M$.

Since $\phi$ is a Lie derivation with respect to the second component, we have
$$
\begin{aligned}
0=\phi(a,[b,e])&=[\phi(a,b),e]+[b,\phi(a,e)]\\
&=\phi(a,b)e-e\phi(a,b)+b\phi(a,e)-\phi(a,e)b
\end{aligned}
$$
for all $a\in A,b\in B$. Multiplying the above equation by $f$ on the left side
and $f$ on the right side, we have $b\phi(a,e)f=f\phi(a,e)b$, and then 
$$
f\phi(a,e)f\in \mathcal{Z}(B)   \eqno(3.9)
$$
for all $a\in A, b\in B$. Combining $(3.7)$ and $(3.9)$
together with the faithfulness of $A$-left module $M$, one can obtain
$$
e\phi(a,e)e\oplus f\phi(a,e)f\in \mathcal{Z}(\mathcal{T})      \eqno(3.10)
$$
for all $a\in A, m\in M$.

 Similarly, we have
$$
e\phi(a,f)e\oplus f\phi(a,f)f\in \mathcal{Z}(\mathcal{T})      \eqno(3.11)
$$
for all $a\in A, m\in M$.

Since $\phi$ is a Lie derivation with respect to the first component, we have
$$
\begin{aligned}
0=\phi(0,m)=\phi([a,e],m)&=[\phi(a,m),e]+[a,\phi(e,m)]\\
&=\phi(a,m)e-e\phi(a,m)+a\phi(e,m)-\phi(e,m)a
\end{aligned}  \eqno(3.12)
$$
for all $a\in A, m\in M$. Multiplying in $(3.12)$ by $e$ on the left side
and by $f$ on the right side, we have
$$
e\phi(a,m)f=a\phi(e,m)f                                       \eqno(3.13)
$$
for all $a\in A, m\in M$. Considering $(3.8)$ and $(3.13)$, we conclude that
$$
\phi(a,m)=a\phi(e,m)f\in M
$$
for all $a\in A, m\in M$. In analogous manner, one can check that
$$
\begin{aligned}
\phi(m,a)&=a\phi(m,e)f~\text{and}~
e\phi(e,a)e \oplus f\phi(e,a)f\in \mathcal{Z}(\mathcal{T})~\text{and}~ e\phi(f,a)e\oplus f\phi(f,a)f\in \mathcal{Z}(\mathcal{T});\\
\phi(b,m)&=e\phi(f,m)b~\text{and}~
e\phi(b,e)e\oplus f\phi(b,e)f\in \mathcal{Z}(\mathcal{T})~\text{and}~e\phi(b,f)e\oplus f\phi(b,f)f\in \mathcal{Z}(\mathcal{T});\\
\phi(m,b)&=e\phi(m,f)b~\text{and}~
e\phi(e,b)e\oplus f\phi(e,b)f\in \mathcal{Z}(\mathcal{T})~\text{and}~e\phi(f,b)e\oplus f\phi(f,b)f\in \mathcal{Z}(\mathcal{T})
\end{aligned}    \eqno(3.14)
$$
for all $a\in A, b\in B, m\in M$.

 Let us see a mapping $\mathfrak{h}:M\rightarrow M$ be defined by $\mathfrak{h}(m)=e\phi(e,m)f$
for all $m\in M$, then $\mathfrak{h}$ is a bimodule homomorphism as a left $A$-module and also right $B$-module
for all $m\in M$.

Namely, for arbitrary $a\in A, b\in B, m\in M$,
since $\phi$ is a Lie derivation with respect to the first argument, in light of the relation $(3.10)$ and $(3.14)$
we can have
$$
\begin{aligned}
\mathfrak{h}(am)&=e\phi(e,am)f\\
&=e\phi(e,[a,m])f\\
&=e([\phi(e,a),m]+[a,\phi(e,m)])f\\
&=e(\phi(e,a)m-m\phi(e,a)f+a\phi(m,e)-\phi(m,e)a)f\\
&=m(\eta(e\phi(e,a)e)-\phi(e,a)f)+a\phi(m,e)f\\
&=a\phi(m,e)f\\
&=a\mathfrak{h}(m)
\end{aligned}
$$
and
$$
\begin{aligned}
\mathfrak{h}(mb)&=e\phi(e,mb)f&\\
&=e\phi(e,[m,b])f\\
&=e([\phi(e,m),b]+[m,\phi(e,b)])\\
&=e(\phi(e,m)b-b\phi(e,m)+m\phi(e,b)-\phi(e,b)m)f\\
&=e\phi(e,m)b+m\phi(e,b)f-e\phi(e,b)m\\
&=e(\phi(e,m)b+m(f\phi(e,b)f-\eta(e\phi(e,b)e))f\\
&=e\phi(e,m)b\\
&=\mathfrak{h}(m)b.
\end{aligned}
$$

The assumption $\text{(iii)}$ implies that the bimodule homomorphism
$\mathfrak{h}$ is of the \textbf{standard form}
$$
\mathfrak{h}(m)=a_0m+mb_0 =e\phi(e,m)f
$$
for some $a_0\in \mathcal{Z}(A)$ and $b_0\in \mathcal{Z}(B)$ and all $m\in M$.
Now we use the assumption $\text{(i)}$ to
see that $a_0\in \pi_{A}(\mathcal{Z}(\mathcal{T}))$ and $b_0\in \pi_{B}(\mathcal{Z}(\mathcal{T}))$.
We may write
$$
\mathfrak{h}(m)=\phi(e,m)=e\phi(e,m)f=(a_0+\eta^{-1}(b_0))m=\alpha_0m
$$
for all $m\in M$, where $\alpha_0=a_0+\eta^{-1}(b_0)\in \pi_{A}(\mathcal{Z}(\mathcal{T}))$.

Likely, one can define a mapping $\mathfrak{g}:M\rightarrow M$ defined by $\mathfrak{g}(m)=e\phi(m,e)f$ for all
$m\in M$, which is a bimodule homomorphism as a left $A$-module and also right $B$-module.
So there exists $\beta_0\in \pi_{A}(\mathcal{Z}(\mathcal{T}))$ so that $\phi(m,e)=\beta_0 m$.

Let us next show that
$$
\mathfrak{h}(m)=\alpha_0m=-\mathfrak{g}(m), ~\text{i.e.,} ~ e\phi(e,m)f==\alpha_0m=-e\phi(m,e)f
$$
for all $m\in M$. We need to prove that $\alpha_0+\beta_0=0$.

According to the assumption $\text{(ii)}$,
we may assume that $A$ is a noncommutative algebra. Choose $a,a^\prime\in A$ such that
$[a,a^\prime]\neq 0$. Since $\phi(e,m)=\alpha_0m$ and $\phi(m,e)=\beta_0m$, we by Lemma \ref{xxsec3.1} get
we have
$$
[\varphi(a,a^\prime),[e,m]]+[\varphi(a,e),[a^\prime,m]]=[\varphi(m,e),[a,a^\prime]]+[\varphi(m,a^\prime),[a,e]]
$$
for all $a,a^\prime\in A, m\in M$. In view of the relation $(3.10)$, we have
$$
[\varphi(a,a^\prime),m]=-[a,a^\prime]\varphi(m,e)f =-[a,a^\prime]\alpha_0 m,                   \eqno{(3.15)}
$$
for all $a,a^\prime\in A, m\in M$.

Adopting similar methods and using Lemma \ref{xxsec3.1}, we have
$$
[\varphi(a,a^\prime),[m,e]]+[\varphi(a,m),[a^\prime,e]]=[\varphi(e,m),[a,a^\prime]]+[\varphi(e,a^\prime),[a,m]]
$$
for all $a,a^\prime\in A, m\in M$. In view of the relation $(3.10)$, we obtain
$$
[\varphi(a,a^\prime),m]=[a,a^\prime]\varphi(e,m)f=[a,a^\prime]\beta_0 m                         \eqno{(3.16)}
$$
for all $a,a^\prime\in A, m\in M$.

It follows from the equalities $(3.15)$ and $(3.16)$ yields $(\alpha_0+\beta_0)[a,a^\prime]m=0$ for all $m\in M$.
The faithfulness of the left $A$-module $M$
now implies $(\alpha_0+\beta_0)[a,a^\prime]=0$. Since $[a,a^\prime]\neq 0$ we conclude using the condition \text{(iv)}
that $\alpha_0+\beta_0=0$. Considering $\alpha_0+\beta_0=0$ and $\varphi(f,m)+\varphi(e,m)=0$ together with $\varphi(m,e)+\varphi(m,f)=0$, we see that
$$
\varphi(m,f)=\alpha_0 m=-\varphi(f,m)
$$
for all $m\in M$.

Let $a\in A$, and $m\in M$ be arbitrary elements, we can achieve
$$
\varphi(a,m)=a\varphi(e,m)f=\alpha_0 am.
$$
This proves the first equality. The other three equations can be proven in an analogous manner.
\end{proof}

\begin{lemma}\label{xxsec3.5}
With notations as above, we have
\begin{enumerate}
  \item [(1)]
$\phi(a,b)=e\phi(a,b)e-a\phi(e,e)b+f\phi(a,b)f$,
where $e\phi(a,b)e\oplus f\phi(a,b)f \in \mathcal{Z}(\mathcal{T})$;
  \item [(2)]
  $\phi(b,a)=e\phi(b,a)e-a\phi(f,f)b+f\phi(b,a)f$,
where $e\phi(b,a)e\oplus f\phi(b,a)f \in \mathcal{Z}(\mathcal{T})$
\end{enumerate}
for all $a\in A, b\in B$.
\end{lemma}

\begin{proof}
  \item [(1)]Since $\phi$ is a Lie derivation with respect to the first component, we have
$$
\begin{aligned}
0&=\phi([e,a],b)\\
&=[e,\phi(a,b)]+[\phi(e,b),a]\\
&=e\phi(a,b)-\phi(a,b)e+\phi(e,b)a-a\phi(e,b)
\end{aligned}
$$
for all $a\in A, b\in B$. Multiplying the above equation by $e$ on the left side
and by $f$ on the right side, we can obtain
$$
e\phi(a,b)f=a\phi(e,b)f            \eqno{(3.17)}
$$
for all $a\in A, b\in B$.

Since $\phi$ is a Lie derivation with respect to the second component, one can have
$$
\begin{aligned}
0&=\phi([b_1,a],b_2)\\
&=[b_1,\phi(a,b_2)]+[\phi(b_1,b_2),a]\\
&=b_1\phi(a,b_2)-\phi(a,b_2)b_1+\phi(b_1,b_2)a-a\phi(b_1,b_2)
\end{aligned}
$$
for all $b_1, b_2\in B, a\in A$. Multiplying the above equation by $e$ on the left side
and by $e$ on the right side, we can obtain $e\phi(b_1,b_2)a=a\phi(b_1,b_2)e$ for all $b_1, b_2\in B$.
And then
$$
e\phi(b_1,b_2)e \in \mathcal{Z}(A)   \eqno{(3.18)}
$$
for all $b_1, b_2\in B$. Multiplying the above equation by $e$ on the left side
and by $f$ on the right side, we can obtain
$$
e\phi(a,b_2)b_1=-a\phi(b_1,b_2)f;    \eqno{(3.19)}
$$
for all $b_1, b_2\in B, a\in A$. Multiplying the above equation by $f$ on the left side
and by $f$ on the right side, we can obtain the relation $b_1\phi(a,b_2)f=f\phi(a,b_2)b_1$
for all $a\in A, b_1, b_2\in B$. And then
$$
f\phi(a,b_2)f\in \mathcal{Z}(B)    \eqno{(3.20)}
$$
for all $a\in A, b_1, b_2\in B$.

Similarly, Since $\phi$ is a Lie derivation with respect to the second component, we have
$$
\begin{aligned}
0&=\phi(a_1,[b,a_2])\\
&=[\phi(a_1,b),a_2]+[b,\phi(a_1,a_2)]\\
&=\phi(a_1,b)a_2-a_2\phi(a_1,b)+b\phi(a_1,a_2)-\phi(a_1,a_2)b
\end{aligned}
$$
for all $a_1, a_2\in A, b\in B$. Multiplying the above equation by $e$ on the left side
and by $e$ on the right side, one can check $e\phi(a_1,b)a_2=a_2\phi(a_1,b)e$, ie.,
$$
e\phi(a_1,b)e \in \mathcal{Z}(A)                      \eqno{(3.21)}
$$
for all $a_1, a_2\in A, b\in B$. Multiplying the above equation by $f$ on the left side
and by $f$ on the right side, we can obtain $b\phi(a_1,a_2)f=f\phi(a_1,a_2)b$, and then
$$
f\phi(a_1,a_2)f\in \mathcal{Z}(B)              \eqno{(3.22)}
$$
for all $a_1, a_2\in A, b\in B$. Multiplying the above equation by $e$ on the left side
and $f$ on the right side, we can obtain
$$
a_2\phi(a_1,b)f=-e\phi(a_1,a_2)b                           \eqno{(3.23)}
$$
for all $a_1, a_2\in A, b\in B$.  Combining $(3.19)$ and $(3.23)$ with the conclusion $(2)$
coming from Lemma \ref{xxsec3.3}, we have
$$
e\phi(a,b)f=-a\phi(e,e)b
$$
for all $a\in A, b\in B$. We therefore have
$$
\phi(a,b)=e\phi(a,b)e-a\phi(e,e)b+f\phi(a,b)f
$$

Now, we prove the following relation
$$
e\phi(a,b)e\oplus f\phi(a,b)f \in \mathcal{Z}(\mathcal{T})
$$
for all $a\in A, b\in B$.

Namely, according to Lemma \ref{xxsec3.1}, we receive the relation
$$
[\phi(a,b),[m,e]]+[\phi(a,m),[e,b]]=[\phi(e,b),[m,a]]+[\phi(e,m),[a,b]]
$$
for all $a\in A, b\in B, m\in M$. In view of the relation
$e\phi(e,b)e\oplus f\phi(e,b)f \in \mathcal{Z}(\mathcal{T})$, we can obtain
$$
\begin{aligned}
e\phi(a,b)m-m\phi(a,b)f&=-e\phi(e,b)am+am\phi(e,b)f\\
&=(\eta^{-1}(f\phi(e,b)f)-e\phi(e,b)e)am=0
\end{aligned}
$$
for all $a\in A, b\in B, m\in M$. Using the relations $(3.20), (3.21)$
and the faithfulness of left $A$-module $M$, we have
$$
e\phi(a,b)e\oplus f\phi(a,b)f\in \mathcal{Z}(\mathcal{T})
$$
for all $a\in A, b\in B$.

\item [(2)]
By an analogous manner of $(1)$, we have
$$
\phi(b,a)=e\phi(b,a)e-a\phi(f,f)b+f\phi(b,a)f,
$$
where $e\phi(b,a)e\oplus f\phi(b,a)f \in \mathcal{Z}(\mathcal{T})$
for all $a\in A, b\in B$.

\end{proof}

\begin{lemma}\label{xxsec3.6}
With notations as above, we have
$$\phi(m,n)=0$$
for all $m,n\in M$.
\end{lemma}

\begin{proof}
For all $m,n\in M$, we have
$$
\begin{aligned}
\varphi(m,n)&=\varphi([e,m],n)\\
&=[\varphi(e,n),m]+[e,\varphi(m,n)]\\
&=\varphi(e,n)m-m\varphi(e,n)+e\varphi(m,n)+\varphi(m,n)e.
\end{aligned}
$$
Multiplying the above equation by $e$ on the left and by $e$ on the right side,
we have $e\varphi(m,n)e=0$ for all $m\in M, n\in N$.
Similarly, one can obtain $f\varphi(m,n)f=0$ for all $m\in M, n\in N$.
By invoking to above relations, we immediately see that
$$
\varphi(m,n)=e\varphi(m,n)f \in M             \eqno{(3.24)}
$$
for all $m\in M, n\in N$.

Fix a element $m\in M$, then the mapping $\mathfrak{k}:M\rightarrow M$ defined
by $\mathfrak{k}(m)=\varphi(m,n)=e\varphi(m,n)f$ for all $n\in M$ is a bimodule homomorphism
as a left $A$-module and also right $B$-module.

In fact, using Lemma \ref{xxsec3.4} and $(3.24)$, we have
$$
\begin{aligned}
\mathfrak{k}(an)&=e\varphi(m,[a,n])f\\
&=e([\varphi(m,a),n]+[a,\varphi(m,n)])f\\
&=e(\varphi(m,a)n-n\varphi(m,a)+a\varphi(m,n)-\varphi(m,n)a)f\\
&=a\varphi(m,n)f\\
&=a\mathfrak{k}(n)
\end{aligned}
$$
and
$$
\begin{aligned}
\mathfrak{k}(nb)&=e\varphi(m,nb)f\\
&=\varphi(m,[n,b])\\
&=e([\varphi(m,n),b]+[n,\varphi(m,b])\\
&=e(\varphi(m,n)b-b\varphi(m,n)+n\varphi(m,b)-\varphi(m,b)n)f\\
&=e\varphi(m,n)b\\
&=\mathfrak{k}(n)b
\end{aligned}
$$
for all $b\in B, m, n\in M$.
For fixing $m\in M$, it follows from the assumption \text{(iii)} that there exists
$\alpha_m\in \mathcal{Z}(A)$ such that
$$
\varphi(m,n)=k(n)=\alpha_m n  ~\text{for all} ~n\in M.   \eqno{(3.25)}
$$

Without loss of generality, we might assume that that $A$ is a noncommutative algebra, and let
$a,a^\prime\in A$ be fixed elements such that $[a,a^\prime]\neq 0$. Using Lemma \ref{xxsec3.1}, we have
$$
[\varphi(a,a^\prime), [n,m]]+[\varphi(a,n), [a^\prime,m]]=[\varphi(m,n),[a,a^\prime]]+[\varphi(m,a^\prime), [a,m]],
$$
for all $m, n\in M$. Using $(3.24)$ and $(3.25)$, we may write
$$
0=[\varphi(m,n),[a,a^\prime]]=[a,a^\prime]\varphi(m,n)=[a,a^\prime]\alpha_m n
$$
for all $m,n\in M$. The faithfulness of the left $A$-module implies $[a,a^\prime]\alpha_m=0$ for every $m\in M$.
In view of the assumption $\text{(iv)}$, we obtain that $\alpha_m=0$ for all $m\in M$.
We therefore say that $\varphi(m,n)=0$ for all $m,n\in M$.
\end{proof}

\begin{lemma}\label{xxsec3.7}
With notations as above, we have
\begin{enumerate}
  \item [(1)] For arbitrary $a_1,a_2\in A$, we have $$
\begin{aligned}
\phi(a_1,a_2)&=e\phi(a_1,a_2)e+a_1a_2\phi(e,e)f+f\phi(a_1,a_2)f\\
&=e\phi(a_1,a_2)e+a_2a_1\phi(e,e)f+f\phi(a_1,a_2)f,
\end{aligned}
$$
where $f\phi(a_1,a_2)f \in \mathcal{Z}(A)$ and $
e\phi(a_1,a_2)e=\eta^{-1}(f\phi(a_1,a_2)f)-\alpha_0[a_1,a_2]
$;
  \item [(2)]For arbitrary $a_1,a_2\in A$, we have $$
\begin{aligned}
\phi(b_1,b_2)&=e\phi(b_1,b_2)e+e\phi(e,e)b_1b_2+f\phi(b_1,b_2)f\\
&=e\phi(b_1,b_2)e+e\phi(e,e)b_1b_2+f\phi(b_1,b_2)f,
\end{aligned}
$$
where $e\phi(b_1,b_2)e \in \mathcal{Z}(A)$ and
$
f\phi(b_1,b_2)f=\eta(e\phi(b_1,b_2)e)-\eta(\alpha_0)[b_1,b_2]
$
for all $b_1,b_2\in B$.
\end{enumerate}
\end{lemma}

\begin{proof}
The conclusion $(1)$ and conclusion $(2)$ can be obtained by the similar ways.
For the sake of conciseness, we now only prove conclusion $(1)$.

\textbf{(1).} Combining $(3.19)$ and $(3.23)$ together with the relation $e\phi(1,x)f=0$ in Lemma \ref{xxsec3.3}, we obtain
$$
e\phi(a_1,a_2)f=-a_2a_1\phi(e,f)f=a_2a_1\phi(e,e)f.    \eqno{(3.26)}
$$
for all $a_1,a_2\in A, x\in T$. In similar methods,
we have
$$
e\phi(a_1,a_2)f=a_1a_2\phi(e,e)f                       \eqno{(3.27)}
$$ for all $a_1,a_2\in A$.
In view of above relations $((3.26))$ and $(3.27)$, we arrive at
$$
e\phi(a_1,a_2)f=a_2a_1\phi(e,e)f=a_1a_2\phi(e,e)f               \eqno{(3.28)}
$$
for all $a_1,a_2\in A$.

According to Lemma \ref{xxsec3.1}, we have
$$
[\phi(a_1,a_2),[e,m]]+[\phi(a,e),[m,a_2]]=[\phi(m,a_2),[e,a]]+[\phi(m,e),[a,a_2]],
$$
for all $a_1, a_2, m\in M$. In view of the relation , we receive at
$$
[\phi(a_1,a_2),m]-[\phi(a,e),a_2m]=-[a,a_2]\phi(m,e)   \eqno{(3.29)}
$$
for all $a_1,a_2\in A, m\in M$
Combining $(3.29)$ and $(3.10)$ together with Lemma \ref{xxsec3.4}, we can achieve
$$
(e\phi(a_1,a_2)e-\eta^{-1}(f\phi(a_1,a_2)f)-\alpha_0[a_1,a_2])m=0
$$
for all $a_1, a_2, m\in M$. The faithfulness
of left $A$-module $M$ now implies
$$
e\phi(a_1,a_2)e=\eta^{-1}(f\phi(a_1,a_2)f)+\alpha_0[a,a_2]
$$
for all $a_1,a_2\in A$.

Based on the above process, one can check that
$$
\phi(a_1,a_2)=\eta^{-1}(f\phi(a_1,a_2)f)+f\phi(a_1,a_2)f+a_1a_2\phi(e,e)f+\alpha_0[a,a_2]
$$
for all $a_1,a_2\in A$.

\textbf{(2).} In analogous manner of $(1)$, we have
$$
\phi(b_1,b_2)=e\phi(b_1,b_2)e+e\phi(e,e)b_1b_2+f\phi(b_1,b_2)f,
$$
where
$$
f\phi(b_1,b_2)f=\eta^{-1}(e\phi(b_1,b_2)e)+\eta(\alpha_0)[b_1,b_2]
$$
for all $ b_1,b_2\in B$.
\end{proof}

Let us now give the proof of our main theorem.
\vspace{2mm}

{\noindent}{\bf Proof of Theorem \ref{xxsec3.2}.}
\vspace{2mm}

At the end, in order to get the final main theorem, we need to obtain the following equation:
$$
aa^\prime\phi(e,e)f-a\phi(e,e)b^\prime-a^\prime\phi(e,e)b+e\phi(e,e)b^\prime=[x,[y,\phi(e,e)]]
$$
for all $x=a+m+b$ and $y=a^\prime+m^\prime+b^\prime$ and all $a,a^\prime\in A, b,b^\prime\in B, m,m^\prime\in M$.

In fact, let
$x=\left[
\smallmatrix
a & m\\
 & b\\
\endsmallmatrix
\right]\in \mathcal{T}$ and $y=\left[
\smallmatrix
a^\prime & m^\prime\\
 & b^\prime\\
\endsmallmatrix
\right]\in \mathcal{T}$, Taking into account the relation $e\phi(e,e)e\oplus f\phi(e,e)f\in \mathcal{Z}(\mathcal{T})$, we arrive at
$$
\begin{aligned}
&[x,[y,\phi(e,e)]]\\
&=[\left[
\smallmatrix
a & m\\
 & b\\
\endsmallmatrix
\right],[\left[
\smallmatrix
a^\prime & m^\prime\\
 & b^\prime\\
\endsmallmatrix
\right],\left[
\smallmatrix
e\phi(e,e)e & e\phi(e,e)f\\
 & f\phi(e,e)f\\
\endsmallmatrix
\right]]]    \\
&=[\left[
\smallmatrix
a & m\\
 & b\\
\endsmallmatrix
\right],\left[
\smallmatrix
0 & a^\prime\phi(e,e)f-e\phi(e,e)b^\prime\\
 & 0\\
\endsmallmatrix
\right]]\\
&=\left[
\smallmatrix
0 & a(a^\prime\phi(e,e)f-e\phi(e,e)b^\prime)-(a^\prime\phi(e,e)f-e\phi(e,e)b^\prime)b\\
 & 0\\
\endsmallmatrix
\right]\\
&=\left[
\smallmatrix
0 & aa^\prime\phi(e,e)f-a\phi(e,e)b^\prime-a^\prime\phi(e,e)b+e\phi(e,e)b^\prime\\
 & 0\\
\endsmallmatrix
\right],
\end{aligned}
$$
for all $a, a^\prime\in A, b, b^\prime\in B,m, m^\prime\in M$

Let $x=a+m+b$ and $y=a^\prime+m^\prime+b^\prime$, according to the bilinearity of the mapping $\phi$,
we get the following decomposition form
$$
\begin{aligned}
\phi(x,y)=&\phi(a,a^\prime)+\phi(a,m^\prime)+\phi(a,b^\prime)\\
&+\phi(m,a^\prime)+\phi(m,m^\prime)+\phi(m,b^\prime)\\
&+\phi(b,a^\prime)+\phi(b,m^\prime)+\phi(b,b^\prime)\\
&=\alpha_0[a,a^\prime]-\alpha_0am^\prime+\alpha_0a^\prime m+\alpha_0 mb^\prime-\alpha_0 m^\prime b+\eta(\alpha_0)[b,b^\prime]\\
&+\eta^{-1}(f\phi(a_1,a_2)f)+f\phi(a_1,a_2)f+e\phi(b_1,b_2)e+\eta(e\phi(b_1,b_2)e)\\
&+e\phi(a,b)e+f\phi(a,b)f+e\phi(b,a)e+f\phi(b,a)f\\
&+aa^\prime\phi(e,e)f-a\phi(e,e)b^\prime-a^\prime\phi(e,e)b+e\phi(e,e)b^\prime\\
&=\lambda_0[x,y]+[x,[y,\phi(e,e)]]+\mu(x,y)
\end{aligned}
$$
where $\mu:\mathcal{T}\times \mathcal{T }\rightarrow \mathcal{Z}(\mathcal{T})$ is a central mapping such that
$$
\begin{aligned}
\mu(x,y)&=\eta^{-1}(f\phi(a_1,a_2)f)+f\phi(a_1,a_2)f+e\phi(b_1,b_2)e+\eta(e\phi(b_1,b_2)e)\\
&+e\phi(a,b)e+f\phi(a,b)f+e\phi(b,a)e+f\phi(b,a)f\\
&=[\left[
\smallmatrix
\eta^{-1}(f\phi(a_1,a_2)f)+e\phi(b_1,b_2)e+e\phi(a,b)e+e\phi(b,a)e & 0\\
 & f\phi(a_1,a_2)f+\eta(e\phi(b_1,b_2)e)+f\phi(a,b)f+f\phi(b,a)f\\
\endsmallmatrix
\right]\\
&\in \mathcal{Z}(\mathcal{T})
\end{aligned}
$$ for all $a,a^\prime\in A, b,b^\prime\in B, m,m^\prime\in M$ .

As a direct corollary of Theorem \ref{xxsec3.2}, we get Lie biderivations of the $(block)$
upper triangular matrix algebra and the nest algebra following from the main theorem.

\begin{corollary}\label{xxsec3.4}
Let $C$ be a commutative domain with identity. If $n\geq 3$, then
each Lie biderivation of $(block)$ upper triangular matrix algebra
$B^{\bar{k}}_{n}(C)$ is the sum of an extremal biderivation and an inner biderivation.
In particular, every biderivation of upper triangular matrix algebra $T_n(C)$ is
the sum of an extremal biderivation and an inner biderivation and central mapping.
 \end{corollary}

\begin{corollary}\label{xxsec3.5}
Let $\mathcal{N}$ be a nest of a Hilbert space $H$, where dim$H\geq 3$.
Then each Lie biderivation $\varphi$ of nest algebra $\mathcal{T}(\mathcal{N})$ is the sum of an
extremal biderivation and an inner biderivation and and central mapping.
\end{corollary}



\addtocontents{toc}{
    \protect\settowidth{\protect\@tocsectionnumwidth}{}%
    \protect\addtolength{\protect\@tocsectionnumwidth}{0em}}

\end{document}